\documentclass[reqno]{amsart}
\usepackage{graphicx}
\usepackage{amssymb}
\usepackage{amsfonts}
\usepackage{indentfirst}
\usepackage{amssymb,amsmath,amsthm}
\usepackage{latexsym,bm}
\usepackage{graphicx}
\usepackage{times}
\usepackage{amsfonts}
\usepackage{indentfirst}
\usepackage[colorlinks,linkcolor=black,anchorcolor=blue,citecolor=blue]{hyperref}
\usepackage[dvipsnames]{pstricks}
\newtheorem{theorem}{Theorem}[section]
\newtheorem{lemma}[theorem]{Lemma}
\newtheorem{proposition}[theorem]{Proposition}

\numberwithin{equation}{section}

\def\Z{\mathbb{Z}}

\def\Zp{\mathbb{Z}_{p}}
\def\Qp{\mathbb{Q}_{p}}

\def\Q{\mathbb{Q}}

\def\P{\mathbb{P}^1(\mathbb{Q}_{p})}
\def\PP{\mathbb{P}^1(\mathbb{Q}_{2})}

\def\Z2{\mathbb{Z}_{2}}
\def\m2#1{\ ({\rm mod} \ 2^{#1})}

\begin{document}

\title{Rational map $ax+1/x$ on  the projective line over $\mathbb{Q}_{2}$ }

\author{Shilei Fan}

\address{School of Mathematics and Statistics, Hubei Key Laboratory of Mathematical Sciences, Central  China Normal University,  Wuhan, 430079, China \&\& Aix-Marseille Universit\'e, Centrale Marseille, CNRS, Institut de Math\'ematiques de Marseille, UMR7373, 39 Rue F. Joliot Curie 13453, Marseille, France}
\email{slfan@mail.ccnu.edu.cn}

\author{Lingmin Liao}
\address{LAMA, UMR 8050, CNRS,
Universit\'e Paris-Est Cr\'eteil Val de Marne, 61 Avenue du
G\'en\'eral de Gaulle, 94010 Cr\'eteil Cedex, France}
\email{lingmin.liao@u-pec.fr}

\thanks{The research of S. L. FAN on this project  has received funding from the European  Research Council(ERC) under the European Union's Horizon 2020 research  and innovation programme under grant agreement No 647133(ICHAOS). It has also been  supported by NSF of China (Grant No.s 11401236 and 11471132)}
\begin{abstract}
The dynamical structure of the rational map $ax+1/x$ on the projective line $\PP$ over the field $\mathbb{Q}_2$ of $2$-adic numbers, 
%is studied in this note.
%For each Chebyshev polynomial, the dynamical structure
 is fully described. %by showing all its minimal subsystems and their attracting basins.

\end{abstract}
%\thanks{S.F. was partially supported by NSF of China (Grant No. 11231009) and CNRS program (PICS No.5727), L.L. was partially supported by 12R03191A - MUTADIS (France) and the project PHC Orchid of MAE and MESR of France}
 %{\it Minist\`eres des Affaires \'Etrang\`eres} and {\it Minist\`eres de l'Enseignement Sup\'erieur et de la Recherche}
\maketitle

%% main text

\section{Introduction}

For a prime number $p$,  let $\mathbb{Q}_p$ be the field of $p$-adic numbers and $\P$ be the projective line over $\Qp$. Recently, dynamical systems on $\mathbb{Q}_p$ and $\P$ have been attracting much attention. See, for example, \cite{Anashin-Khrennikov-AAD,Baker-Rumely-book,SilvermanGTM241} and their bibliographies therein. 

As first natural examples, polynomials and rational maps  of coefficients in $\mathbb{Q}_p$ have been widely studied. 
The dynamical structure of polynomials of coefficients in the ring $\Zp$ of $p$-adic integers is usually described by showing all the minimal subsystems. One can find such dynamical structure results for square mapping on $\Zp$ in \cite{FLSq}, for quadratic polynomials on $\Z2$ in \cite{FL11} and for Chebyshev polynomials on $\Z2$ in \cite{FLCh}.  With respect to the spherical metric on $\Qp$, rational maps of good reduction with degree at least $2$ have similar dynamical structure \cite{FFLW2016}.
On the other hand, the polynomial ${x^p-x \over p}$ on $\Zp$ is proved to be topologically conjugate to the full shift on the symbolic space with $p$ symbols and thus exhibits chaos (\cite{WS98}). This chaotic property has also been studied for quadratic polynomials in $\Qp$ (\cite{TVW89, DSV06}) and for some general expanding polynomials in $\Qp$ (\cite{FLWZ07}). 

In \cite{FFLW2014}, the dynamical structure of rational maps of degree one on $\P$ is completely depicted. 
 %It turns out that such $p$-adic dynamical systems are quite different to the dynamical systems on Euclidean spaces ().  
However, the dynamical behaviors of higher degree rational maps are far from clear. In \cite{FL16}, we have started an attempt with a family of rational maps of degree $2$. That is 
 \begin{align}
\phi(x)=ax+\frac{1}{x}, \ a\in \Q_p.
\end{align}
Except for one subcase ($|a|_p<1, \sqrt{a}\in \Qp$), the global picture of the dynamical structure of $\phi$ on $\P$ for $p\geq 3$ was given in \cite{FL16}. In this paper, we deal with the case $p=2$ and we successfully show the global dynamical structure of $\phi$ for all cases. %Different to the case $p\geq 3$, there might exist Julia set of positive Haar measure. 
We remark that the case $p=2$ is usually different to the case $p\geq 3$ and has much more applications in computer science.

\smallskip

To describe the dynamical structure of a dynamical system, one may decompose the space into Fatou set and Julia set, on which the dynamical behaviors are quite different. Let $\phi :X\to X$ be a continuous map from a  metric space to  itself.       The {\bf Fatou set} $\mathcal{F}_{\phi}$ of $\phi$ is the set of point $z\in X$ where  the iteration family $\{\phi^n\}_{n\geq 1}$  is equicontinuous  on some neighborhood of  $z$.   The {\bf Julia set} $\mathcal{J}_{\phi}$ is the complement of the Fatou set in $X$.    

We denote by $|\cdot|_p$  the $p$-adic absolute value on $\Q_p$. For $x_0\in \Q_p$ and $r>0$, let 
$$D(x_0,r):=\{x\in \Q_p: |x-x_0|_p \leq  r\}$$   
be the closed disk centered at $x_0$ with radius $r$,
and let  
$$S(x_0,r):=\{x\in \Q_p: |x-x_0|_p =  r\}$$ 
be the sphere centered at $x_0$ with radius $r$.
The projective line $\mathbb{P}^{1}(\Qp)$ which can be viewed as $\Qp\cup\{\infty\}$ is equipped the spherical metric (see Section \ref{Sec-preliminaries}).  
     
With respect to the spherical metric, a rational map $\phi\in \Q_p(z)$ induces a continuous dynamical system from $\P$ to itself, denoted by $(\P,\phi)$. If the rational map $\phi$ has good reduction (see the definition in \cite{SilvermanGTM241}, page 58, see also \cite{FFLW2016}), then it is $1$-Lipschitz continuous on $\P$ and thus has empty Julia set.
%A $1$-Lipschitz $p$-adic dynamical system can usually be fully described by showing all its minimal subsystems.
%Polynomials with coefficients in the ring $\Zp$ of $p$-adic integers and rational maps with good reductions are two important families of $1$-Lipschitz dynamical systems. 
%In \cite{FL11}, the authors proved  a structure theorem for polynomials of degree
% at least $2$ in $\mathbb{Z}_p[x]$. The same 
 A structure theorem for good reduction maps with degree at least $2$ has already been obtained in \cite{FFLW2016}.

Observe that  $$\phi(x)=ax+\frac{1}{x}=\frac{ax^2+1}{x}.$$ 
In all cases, $\infty$ is a fixed point. If $p=2$ and $|a|_2=1$, the map $\phi$ has good reduction which implies $\mathcal{J}_{\phi}=\emptyset$.
Further, by Theorem 1.2 of \cite{FFLW2016}, we have the following minimal decomposition.
\begin{theorem}\label{minid}
Let $\phi(x)=ax+\frac{1}{x}, a\in \Q_2$.
 If $|a|_2=1$, the dynamical system $(\PP,\phi)$ can be decomposed into 
$$\PP= \mathcal{P} \bigsqcup \mathcal{M} \bigsqcup \mathcal{B}
$$
where $ \mathcal{P}$ is the  finite set  consisting of all periodic points of
$\phi$, $\mathcal{M}= \bigsqcup_i \mathcal{M}_i$ is the union of all (at most countably
many) clopen invariant sets such that each $\mathcal{M}_i$ is a finite union
of balls and each subsystem $\phi: \mathcal{M}_i \to \mathcal{M}_i$ is minimal, and each
point in $\mathcal{B}$ lies in the attracting basin of a periodic orbit or of
a minimal subsystem.
\end{theorem}

To complete the study of the dynamical behaviors of 
 \begin{align}
\phi(x)=ax+\frac{1}{x}, \ a\in \Q_2,
\end{align}
on $\mathbb{P}^1(\Q_2)$, we are left two cases: $|a|_2>1$ and $|a|_2<1$. Solving these two cases is the main content of the present paper. In the rest two cases, we often find subsystems topologically conjugate to the full shift $(\Sigma_2, \sigma)$ or some subshift of finite type $(\Sigma_A, \sigma)$ (see the notation and definitions in Section \ref{Sec-preliminaries}).

When  $|a|_2>1$, we have the following theorem. 
\begin{theorem}\label{thm-2}
%Let $\phi(x)=ax+\frac{1}{x}, a\in \Q_2$.  
If $|a|_2> 1$, the  dynamical system $(\PP,\phi)$ is described  as following.
\begin{itemize}
\item[(i)] If $|a|_2>1$ and   $\sqrt{1-a}\notin\Q_2$, then $\mathcal{F}_{\phi}=\PP$.  Moreover, \[\forall x\in \Q_2, \quad \lim_{n\to\infty} \phi^n(x)=\infty.\]
\item[(ii)] If $|a|_2>1$ and  $\sqrt{1-a}\in\Q_2$, then  
$$\mathcal{J}_\phi=\left\{x\in \Q_2:|\phi^{n}(x)|_2=\frac{1}{\sqrt{|a|_2}} \hbox{  for all } n\geq 0 \right\}$$
and $(\mathcal{J}_\phi,\phi)$ is topologically conjugate to the full shift  $ (\Sigma_{2}, \sigma)$. 
Moreover,  $$\forall x\in \mathcal{F}_\phi, \quad  \lim_{n\to \infty}\phi^n(x)=\infty.$$
\end{itemize} 
\end{theorem}

\smallskip
When  $|a|_2<1$, we set  $X_a:=\left\{x\in \mathbb{Q}_2: \sqrt{|a|_2}<|x|_2<\frac{1}{\sqrt{|a|_2}}\right\}$. We have the following theorem.
\begin{theorem}\label{a<1}
  Assume $|a|_2<1$.   Then 
  $$\mathcal{F}_{\phi}=\bigcup_{i=0}^{\infty}f^{-i}(X_a)$$ and 
  we have a minimal decomposition of the subsystem $(X_a,\phi)$  as stated  in Theorem \ref{minid}. 
  
  For the Julia set $\mathcal{J}_{\phi}$, we distinguish the following three cases.
 \begin{itemize}
\item[(i)]  If $\sqrt{-a}\in \Q_2$, then  $$\mathcal{J}_\phi=\left\{x\in \PP: \phi^{n}(x) \notin  X_a
 \hbox{  for all } n\geq 0 \right\} $$
 and the subsystem
$\left(\mathcal{J}_{\phi}\cap \left(\PP\setminus( S(0,\sqrt{|a|_2})\cap X_a) \right), \phi \right)$ is topologically conjugate to a subshift of finite type $(\Sigma_{A},\sigma),$
where the incidence matrix is
\begin{equation*}\label{equ1}
A=
\begin{cases}
\begin{pmatrix} 0 & 0& 1 & 0   \\ 0&0&1&0\\  0&0&0&1 \\  1&1&0&1 \end{pmatrix}, \ \hbox{ if } |a|_2<1/4,\\
\begin{pmatrix} 0 & 0& 1 & 0 &0  \\ 0&0&1&0&0\\  0&0&0&1&0 \\  0&0&0&1&1 \\ 1&1&0&0&0 \end{pmatrix},\ \hbox{ if } |a|_2=1/4.
\end{cases}
\end{equation*}

%$$A=\begin{pmatrix} 0 & 0& 1 & 0   \\ 0&0&1&0\\  0&0&0&1 \\  1&1&0&1 \end{pmatrix}.$$

\item[(ii)]  If $\sqrt{-a}\notin \Q_2$, $\sqrt{-a}\in \Q_2(\sqrt{-3})$ and $|a|_2=1/4$, then 
$$\mathcal{J}_{\phi}=\{0,\infty \} \bigcup_{k\geq 0} (S(0,2^{2k+1})\cup S(0,2^{-2k-1} )).$$ 
Furthermore,
\begin{equation*}\label{equ1}
\begin{cases}
\phi^{k-1}(x) \in  S(0,2^{3}), \ \hbox{ if } x\in  S(0,2^{2k+1}),\\
\phi^{k}(x) \in  S(0,2^{3}), \ \hbox{ if } x\in  S(0,2^{-(2k+1)});
\end{cases}
\end{equation*}
and the subsystem
$(S(0,2^{3})\cup S(0,2^{-3} )) \cup S(0,2), \phi)$ is topologically conjugate to a subshift of finite type $(\Sigma_{A},\sigma),$
where the incidence matrix is
$$A=\begin{pmatrix} 0 & 0& 1 & 0   \\ 0&0&1&0\\  0&0&0&1 \\  1&1&0&0 \end{pmatrix}.$$ 
\item[(iii)]  Otherwise,   $\mathcal{J}_{\phi}=\{0,\infty\}$.
 \end{itemize}
 
\end{theorem}

\medskip
\section{Preliminaries}\label{Sec-preliminaries}
Consider the field $\mathbb{Q}$ of rational numbers and a prime $p\geq 2$.
Any nonzero rational number $r\in \mathbb{Q}$ can be written as
$r =p^v \frac{a}{b}$ where $v, a, b\in \mathbb{Z}$ and $(p, a)=1$ and $(p, b)=1$
(here $(x, y)$ denotes the greatest common divisor of two integers $x$ and $y$).  We define $v_p(r)=v$ and
$|r|_p = p^{-v_p(r)}$ for $r\not=0$ and $|0|_p=0$.
Then $|\cdot|_p$ is a non-Archimedean absolute value on $\Q$. That means\\
\indent (i)  \ \ $|r|_p\ge 0$ with equality only for $r=0$; \\
\indent (ii) \ $|r s|_p=|r|_p |s|_p$;\\
\indent (iii) $|r+s|_p\le \max\{ |r|_p, |s|_p\}$.\\
The {\it field of $p$-adic numbers} $\mathbb{Q}_p$ is the completion of $\mathbb{Q}$ under the absolute value 
$|\cdot|_p$. 
Actually, any $x\in \Qp$ can be written  uniquely as
$$
   x = \sum_{n=v_p(x)}^\infty a_n p^n    \quad (v_p(x) \in \mathbb{Z}, a_n \in\{0, 1, 2, \cdots, p-1\} \hbox{ and  }a_{v_p(x)}\neq 0).
$$
Here, the integer $v_p(x)$ is called the {\em $p$-valuation} of $x$.

 Any point in the {\it projective line} $\P$ may be given in homogeneous coordinates by a pair
$[x_1 : x_2]$
of points in $\Qp$ which are not both zero. Two such pairs are equal if they differ by an overall (nonzero) factor $\lambda \in \Q_p^*$:
$$[x_1 : x_2] = [\lambda x_1 : \lambda x_2].$$
The field $\Qp$ may be identified with the subset of $\P$ given by
$$\left\{[x : 1] \in \mathbb P^1(\Qp) \mid x \in \Qp\right\}.$$
This subset contains  all points  in $\P$ except one: the point  of infinity, which may be given as
$\infty = [1 : 0].$

The {\it spherical metric} defined on $\P$ is analogous to the standard spherical metric on the
Riemann sphere. If $P=[x_1: y_1]$ and $Q=[x_2 : y_2]$ are two points in $\P$, we define
 $$\rho(P,Q)=\frac{|x_1y_2-x_2y_1|_p}{\max\{|x_1|_{p},|y_1|_{p}\}\max\{|x_2|_{p},|y_{2}|_{p}\}}.$$
Viewing $\P$ as $\Qp\cup\{\infty\}$, the spherical distance of $z_1,z_2 \in \Qp\cup \{\infty\}$, can be described by 
 \begin{align*}
 \rho(z_1,z_2)&=\frac{|z_1-z_2|_{p}}{\max\{|z_1|_{p},1\}\max\{|z_2|_{p},1\}}  \qquad\mbox{if~}z_{1},z_{2}\in \Qp;\\
 \rho(z,\infty)&=\left\{
                    \begin{array}{ll}
                      1, & \mbox{if $|z|_{p}\leq 1$,} \\
                      1/|z|_{p}, & \mbox{if $|z|_{p}> 1$.}
                    \end{array}
                  \right.
 \end{align*}   

Remark that the restriction of the spherical metric on the ring $\Zp:=\{x\in \Q_p, |x|_p\leq 1\}$ of $p$-adic integers is same as the metric induced by the absolute value $|\cdot|_p$.

In what follows, we restrict ourselves to the case $p=2$ only.
We  recall the conditions under which a number in $\Q_2$ has a square root in $\Q_2$, then we present all possible quadratic extensions of $\Q_2$.

%An integer $a\in \mathbb{Z}$ is called a\emph{ quadratic residue modulo} $p$ if the equation $x^{2}\equiv a ~(\!\!\!\!\mod p)$ has a solution $x\in \mathbb{Z}$. The following lemma characterizes those $p$-adic integers which admit a square root in $\Qp$.
\begin{lemma}[\cite{Mahler81}]\label{solution} Let $a$ be a nonzero $2$-adic number with its $2$-adic expansion
$$ a=2^{v_{2}(a)}(a_0+a_{1}\cdot 2+a_{2}\cdot 2^{2}+\cdots)$$
where $a_0=1$ and $0\leq a_j \leq 1 \ (j\geq 1)$. The equation
$x^2=a$ has a solution $x\in \Q_2$ if and only if $v_{p}(a)$ is even and $a_1=a_2=0$.
% the following conditions are satisfied
%\, \indent \ \
%  \item[{\rm (i)}] $v_{p}(a)$ is even;
% \, \indent \ \
%  \item[{\rm (ii)}] $a_1=a_2=0$.
\end{lemma}

%For the case where $\sqrt{\Delta}\notin \Q_2$, we need to study the affine systems on some  quadratic extension of $\Q_2$.
%Two distinct elements $a$ and $a^{\prime}$ of $\Q_2$, neither of which is $0$ or the square of a $2$-adic number, evidently produce the same extension field
%$$\Q_2(\sqrt{a})=\Q_2(\sqrt{a^{\prime}})$$
%if and only if the quotient $a/a^{\prime}$ is the square of a $2$-adic number. Actually, there are $7$ distinct quadratic extensions of $\Q_{2}$.
\begin{lemma}[ {\cite[Theorem 1, p.72]{Mahler81}}]\label{7extension}
\item[{\rm (i)}] For $a, a^{\prime}\in \Q_2\setminus \{0\}$, 
$\Q_2(\sqrt{a})=\Q_2(\sqrt{a^{\prime}})$  if and only if the quotient $a/a^{\prime}$ is a square of a $2$-adic number.
\item[{\rm (ii)}] There are exactly $7$ distinct quadratic extensions of $\Q_2$ which are represented respectively by
$$\Q_2(\sqrt{-1}),\Q_2(\sqrt{2}),\Q_2(\sqrt{-2}),\Q_2(\sqrt{3}),\Q_2(\sqrt{-3}),\Q_2(\sqrt{-6}) \hbox{  and } \Q_2(\sqrt{6}).$$
%If $p\geq 3$, then $\Q_p$ has exactly $3$ distinct quadratic
%extensions: %, and these may be represented by
%$$\Qp(\sqrt{N_p}),\Qp(\sqrt{p}),\Qp(\sqrt{pN_p}),$$
%where $N_p<p$ is the smallest positive integer which is not a quadratic residue modulo $p$.
\end{lemma}

\medskip
%\subsection{Terminology of the theory of dynamical systems and $p$-adic repeller}
Now we recall some standard terminology of the theory of dynamical systems. 
%A {\em dynamical system} is a couple $(X, \phi)$, where $X$ is a topological space and $\phi$ is a continuous transformation from $X$ to itself. A subset $Y\subset X$ is {\em invariant} if $\phi(Y)\subset Y$. For such an invariant subset, let $\phi|_Y$ be the restriction of $\phi$ on $Y$. Then $(Y, \phi|_Y)$ becomes a {\em subsystem}. For simplicity, sometimes we still denote by $\phi$ the restricted transformation $\phi|_Y$. 
%
%A main interest in the theory of dynamical systems is to study the topological properties of the orbit $\{\phi^{ n}(x)\}_{n\geq 1}$ of a point $x\in X$ under the iterations $\phi^{ n}$.

A point $x_0\in X$ is called a {\em fixed point} of $\phi$ if $\phi(x_0)= x_0$. 
For a fixed point $x_0\in \Q_p$, the derivative $\phi^{\prime}(x_0)$ (if exists) is called the multiplier of $x_0$. Remark that the  multiplier is invariant by changing of coordinate.  If $\infty$ is a fixed point, then the multiplier of $\infty$ is $\psi^{\prime}(0)$, where $\psi(x)=\frac{1}{\phi(1/x)}$.
A fixed point is called \emph{attracting, indifferent,} or \emph{repelling} accordingly as the absolute value of its multiplier is less than, equal to, or greater than $1$. Fixed points of multiplier $0$ are called \emph{super attracting}.

%A point $x_0\in X$ is called a {\em fixed point} of $\phi$ if $\phi(x_0)= x_0$. %The set of all fixed points of $\phi$ is denoted by ${\rm Fix}(\phi).$   
%It is called a \emph{periodic point }of $\phi$ if $\phi^{ n}(x_0)=x_0$ for some $n\geq 1$. In this case, $n$ is called a \emph{period} of $x_0$, and the smallest $n$ with this property is called the \emph{exact period} of $x_0$.
%For a  periodic  point $x_0\in \Q_p$ of exact period $n$, the derivative $(\phi^{ n})^{\prime}(x_0)$ (if exists) is called the multiplier of $x_0$. Remark that the  multiplier is invariant by changing of coordinate.  If $\infty$ is a periodic point of period $n$, then the multiplier of $\infty$ is $\psi^{\prime}(0)$, where $\psi(x)=\frac{1}{\phi^{ n}(1/x)}$.
%A  periodic point is called \emph{attracting, indifferent,} or \emph{repelling} accordingly as the absolute value of its multiplier is less than, equal to, or greater than $1$. Periodic points of multiplier $0$ are called \emph{super attracting}.

A subsystem of a dynamical system is {\it minimal} if the orbit of any point in the subspace is dense in the subspace. 

%There is an important class of dynamical systems: symbolic dynamical systems. 
For $m\geq 1$, call the finite set  $\mathcal{A}:=\{0, 1, \cdots, m-1\}$ an alphabet. The infinite product space $\Sigma_m:=\mathcal{A}^{\mathbb{N}}$ is said to be the symbolic space of $m$ symbols. The (left)-shift $\sigma$ on  $\Sigma$ is defined as 
$$\forall x=x_0x_1x_2\cdots \in \Sigma_m, \quad (\sigma(x))_i=x_{i+1}.$$
The couple $(\Sigma_m, \sigma)$ is called a {\em full shift} on $m$ symbols. Let $A$ be a $m\times m$ matrix with entries in $\{0,1\}$ and set
\[
\Sigma_A=\left\{(x_j)_{j\geq 0}\in \Sigma_m: A_{x_ix_{i+1}}=1, \ \forall i\geq 0\right\}.
\]
Then $\Sigma_A$ is an invariant subset of $\Sigma_m$. We thus obtain a subsystem $(\Sigma_A, \sigma)$ called a {\em subshift of finite type}.

A large family of $p$-adic dynamical systems are conjugate to full shift or subshift of finite type. Let $f: X \to \Q_p$ be a map from a compact open set $X$ of $\Q_p$ into $\Q_p$ with $f(X)\supset X$. For such a map $f$, define 
\begin{equation}\label{Julia}
    \mathcal{K}_{f,X} := \bigcap_{n=0}^\infty f^{-n}(X).
\end{equation}
It is clear that $f^{-1}(\mathcal{K}_{f,X} ) = \mathcal{K}_{f,X} $ and then $f(\mathcal{K}_{f,X} ) \subset
\mathcal{K}_{f,X} $. Hence, we have a subsystem $(\mathcal{K}_{f,X},f)$. The subset $\mathcal{K}_{f,X} $ is sometimes called filled Julia set. For a complex dynamical system, the Julia set is the boundary of filled Julia set in $X$. However, it is not true for in $\Q_p$, since $\Q_p$ is not algebraically  closed.
 
 Assume that  $X= \bigsqcup_{i\in I}
D(c_i, {p^{-\tau}})$ can be written as a finite disjoint union of
disks of centers $c_i$ and of the same radius $p^{-\tau}$ (with
some $\tau \in \mathbb{Z}$) such that for each $i\in I$ there is
an integer $\tau_i \in \mathbb{Z}$ such that
\begin{equation}\label{expandingness}
  \forall x, y \in D(c_i, {p^{-\tau}}), \quad  |f(x)-f(y)|_p=p^{\tau_i} |x-y|_p.
\end{equation}
The subsystem $(\mathcal{K}_{f,X} , f)$ is called a {\em $p$-adic weak
 repeller} if all $\tau_i$ in (\ref{expandingness}) are nonnegative, but at least
one is  positive.  

The chaotic  dynamical behavior of a $p$-adic weak repeller is determined by a matrix with entries in $\{0,1\}$.
Denote simply the disk $D(c_i, {p^{-\tau}})$ by $D_i$. For any $i \in I$, let
$$I_i :=\{j\in I: D_j\cap f(D_i)\neq\emptyset \} =\{j\in I: D_j\subset f(D_i) \},
$$
%(the second equality holds because of the expansiveness and of the ultrametric property). 
and define the {\em incidence matrix} $A=(A_{i,
j})_{I\times I}$ of $(\mathcal{J}_f, f)$, by
$$A_{ij}= 1 \ \ \mbox{\rm if} \  j\in I_i; \quad
A_{ij}= 0 \ \ \mbox{\rm otherwise}.
$$
The weak repeller $(\mathcal{K}_{f,X} , f)$ is called {\em transitive} if the matrix $A$ is irreducible.
In \cite{FLWZ07}, the authors proved the following theorem. 
\begin{theorem}[\cite{FLWZ07}, Theorem 1.1]\label{thm-repeller}
Let $(\mathcal{K}_{f,X} , f)$ be a transitive $p$-adic weak repeller in $\Q_p$ with incidence matrix $A$. Then  $(\mathcal{K}_{f,X} , f)$ is topologically conjugate to $(\Sigma_A, \sigma)$, i.e., there exists a continuous bijection $h:\mathcal{K}_{f,X}  \rightarrow \Sigma_A$ such that $h\circ f=\sigma\circ h$. Moreover, the   Julia set of the system $(\mathcal{K}_{f,X} ,f)$ is the whole space $\mathcal{K}_{f,X} $.
 %Moreover, the Julia set of $f$ is the same as filled Julia set, $\mathcal{K}_{f,X} =\mathcal{J}_f$.
 \end{theorem}

%\end{proposition}

More generally,  we can  apply this idea to ``generalized $p$-adic repellers". 
Assume that $f: X \to \Q_p$ is a map from a compact open set $X=\bigsqcup_{i\in I} D(c_i, {p^{-\tau_i}})$ into $\Q_p$  such that 
 $f(X) \supset X$ and  for each $i\in I$ there is
an integer $\gamma_i \in \mathbb{Z}$ such that
\begin{equation}\label{expandingness}
 \forall x, y \in D(c_i, {p^{-\tau_i}}), \quad  |f(x)-f(y)|_p=p^{\gamma_i} |x-y|_p.
\end{equation}
Denote simply the disk $D(c_i, {p^{-\tau_i}})$ by $D_i$.  The subsystem $(\mathcal{K}_{f,X} , f)$  is called a {\em generalized $p$-adic weak repeller} 
if the map  $f: X \to \Q_p$ satisfies  the following two conditions: \\
{\rm (i)} for each $i\in I$, $f(D_i)$ contains at least some $D_j$;\\
{\rm (ii)}  there exists at least one $i\in I$ such that $f(D_i)\supsetneq D_j$.

We define similarly the incidence matrix $A$ associated to $(\mathcal{K}_{f,X},f)$. The generalized $p$-adic weak repeller $(\mathcal{K}_{f,X} , f)$  is transitive if $A$ is transitive.

\begin{theorem}\label{general}
Let $(\mathcal{K}_{f,X} , f)$ be a transitive generalized $p$-adic weak repeller in $\Q_p$ with incidence matrix $A$. Then  $(\mathcal{K}_{f,X} , f)$ is topologically conjugate to $(\Sigma_A, \sigma)$, i.e., there exists a continuous bijection $h:\mathcal{K}_{f,X}  \rightarrow \Sigma_A$ such that $h\circ f=\sigma\circ h$.  Moreover, the   Julia set of the system $(\mathcal{K}_{f,X} ,f)$ is the whole space $\mathcal{K}_{f,X} $.

 %Moreover, $\mathcal{K}_{f,X} =\mathcal{J}_f$.
\end{theorem}
\begin{proof}
Let $\tau=\max_{i\in I}\{  \tau_i\}$ and $X^{g}=\bigsqcup_{i\in I} D(c_i, {p^{-\tau}})$. Take a map  $g: X\to X^{g}$ 
such that $$
   g(x)=   \left\{ \begin{array}{ll}
         c_i, & \mbox{if $x=c_i$ for some $i\in I$;} \\
        c_i+p^{\tau-\tau_i}(x-c_i), & \mbox{if  $x \in D(c_i, {p^{-\tau_i}})$ for some  $i\in I$.} 
       \end{array}
       \right.
$$
Then the system $(X, f)$ is topologically conjugate to  $(X^{g}, f^{g}=g\circ f \circ g^{-1}) $. Note that $X^{g}$ consists of ball of same the radius $p^{-\tau}$.   The condition (i)  implies that  $f^{g}$ is non contracting  on each ball $g(D_i)$ of $X^{g}$ and  the condition (ii)  implies $f^{g}$  is expanding  on at least one ball of $X^{g}$.
 Observing that the associated incidence matrix does not change by the topological conjugacy $g$, we conclude by applying Theorem \ref{thm-repeller}.
\end{proof} 
%For such a map $f$, define its Julia set by
%\begin{equation}\label{Juilia}
%    \mathcal{J}_f = \bigcap_{n=0}^\infty (f^{-1})^{\circ n}(X).
%\end{equation}
%It is clear that $f^{-1}(\mathcal{J}_f) = \mathcal{J}_f$ and then $f(\mathcal{J}_f) \subset
%\mathcal{J}_f$. We thus have a subsystem $(\mathcal{J}_f, f)$. 
%Such a subsystem is called a {\em $p$-adic weak
% repeller} if all $\tau_i$ in (\ref{expandingness}) are nonnegative, but at least
%one is  positive.  

\medskip
\section{Dynamical systems}

Recall that we need to study the rational map $\phi(x)=ax+\frac{1}{x}, a\in \Q_2$ for the rest cases $|a|_2>1$ and $|a|_2<1$.  An easy computation shows that 
$$\phi(x)-\phi(y)=(a-\frac{1}{xy})(x-y)$$
and $$\phi^{\prime}(x)=a-\frac{1}{x^2}.$$

%We can also consider the rational maps 
%\[
%ax+{b \over x}, \quad \text{with} \ a, b\in \Q_2.
%\]
%Remark that if $\sqrt{b}$ exists in $\Q_2$, then $ax+{b \over x}$ is conjugate to $ax+{1 \over x}$ by the conjugacy ${1 \over \sqrt{b}}x$.

\subsection{$|a|_2>1$} %It is obvious that $\infty$ is a fixed point. 
We distinguish two subcases: $\sqrt{1-a} \notin \Q_2$ and $\sqrt{1-a} \in \Q_2.$ 

\smallskip
In the first subcase, every point goes to infinity.
\begin{proposition}\label{prop-1}
Suppose $|a|_2>1$. If $\sqrt{1-a}\notin\Q_2$, then \[\forall x\in \Q_2, \quad \lim_{n\to\infty} \phi^n(x)=\infty.\]
\end{proposition}
\begin{proof}
By the assumption $|a|_2>1$, for all $x\in \Q_2$ such that $|x|_2\geq 1$, we have 
\[|\phi(x)|_2=\left|ax+{1\over x}\right|_2=|ax|_2>|x|_2.\]
Thus the absolute values of the iterations $\phi^n(x)$ are strictly increasing. Hence
\begin{align}\label{attr-basin}
\lim_{n\to\infty} \phi^n(x)=\infty, \quad \text{for all } x\in \Q_2, |x|_2\geq 1.
\end{align}
That is to say, $\{x\in \Q_2: |x|_2\geq 1\}$ is included in the attracting basin of $\infty$.

Now we investigate the points in the ball $\{x\in \Q_2: |x|_2\leq1/2\}$. We partition this ball into two: $$A_1:=\left\{x\in \Q_2 :|x|_2\leq 1/2 \hbox{ and  } |ax|_2\neq {1\over  |x|_2}\right\}, \ A_2:=\left\{x\in \Q_2 :  |ax|_2={1\over  |x|_2}\right\}.$$  

If $x\in A_1$, then 
\[
|\phi(x)|_2=\max\left\{|ax|_2, {1 \over |x|_2}\right\} >1.
\]
Thus by (\ref{attr-basin}), $\phi(x)$ falls in the attracting basin of $\infty$, and $\lim_{n\to\infty} \phi^n(x)=\infty$. 

For $x\in A_2$, we will study separately according to the parity of $v_2(a)$. %show $\phi(x)\in  A_1$, then the conclusion of the proposition follows.
%Hence, it suffices to prove that  $\phi(A_2)\in \Q_2 \setminus A_2$.

 When $v_2(a)$ is odd, $A_2=\emptyset$. 
So, we can conclude that    \[\forall x\in \Q_2, \quad \lim_{n\to\infty} \phi^n(x)=\infty.\]

When $v_2(a)$ is even, we distinguish two cases: (i) $v_2(a)=-2$ and (ii) $v_2(a)<-2$.
Note that $A_2=S(0,1/\sqrt{|a|_2})=D(\sqrt{|a|_2},\frac{1}{4\sqrt{|a|_2}})\cup D(-\sqrt{|a|_2},\frac{1}{4\sqrt{|a|_2}})$.
Let $x\in A_2$. Then we have $x=\sqrt{|a|_2}+y$ or $-\sqrt{|a|_2}+y$ for some $y\in D(0,\frac{1}{4\sqrt{|a|_2}})$. 
Without loss of generality, we assume $x=\sqrt{|a|_2}+y$ with $y\in D(0,\frac{1}{4\sqrt{|a|_2}})$.
So we have 
$$|\phi(x)|_2=\left|\frac{ax^2+1}{x}\right|_2=\frac{|2^{-v_2(a)}a+1+2\cdot a \cdot 2^{-v_2(a)/2}y+a\cdot y^2|_2 }{|x|_2}.$$
(i) $v_2(a)=-2$. By Lemmas \ref{solution} and \ref{7extension}, the assumption $\sqrt{1-a}\notin\Q_2$  implies $$2^{-v_2(a)}a= \pm1 \hbox{ or }-3  ~ \ ({\rm mod} \ 8).$$ 
If  $2^{-v_2(a)}a= 1 \hbox{ or }-3  ~(\!\!\!\!\mod 8),$ then  $$|\phi(x)|_2=\frac{|2^{-v_2(a)}a+1|_2}{|x|_2}=\frac{1}{2|x|_2}=1.$$
If  $2^{-v_2(a)}a= -1~(\!\!\!\!\mod 8),$ then  $$|\phi(x)|_2\leq \frac{1}{8|x|_2}=\frac{1}{4}.$$ So $\phi(x)\in A_1$ and we conclude by the above study of the points in $A_1$.
 \\
(ii) $v_2(a)$ is even and $v_2(a)\neq -2$. By Lemmas \ref{solution} and \ref{7extension} ,$\sqrt{1-a}\notin\Q_2$  if and only if $\sqrt{-a}\notin \Q_2$ which implies  $2^{-v_2(a)}a= 1$ or $\pm 3  ~(\!\!\!\!\mod 8)$. So we have 
$$|\phi(x)|_2=\frac{|2^{-v_2(a)}a+1|_2}{|x|_2}\geq \frac{1}{4|x|_2}\geq 1.$$
By (\ref{attr-basin}), the proof is completed.

%If $x\in A_2$, then $|ax|_2={1\over  |x|_2}$. Thus $|a|_2={1\over |x|_2^2}$, which means that $v_2(a)$ is an even number. 
% We write 
%\[
%-a={1 \over 2^{2k}} + {a_{1-2k} \over 2^{2k-1}} +{a_{2-2k} \over 2^{2k-2}}+\cdots = {1 \over 2^{2k}}(1+2a_{1-2k}+4a_{2-2k}+ \cdots),
%\]
%with $k\geq 2$. Then $\sqrt{1-a}\not\in \Q_2$ which implies that \\
%(i) $a_{1-2k}=1$ or, \\
%(ii) $a_{1-2k}=0$ and $a_{2-2k}=1$.
%
%Recall that we want to investigate all the points $x\in A_2$, i.e., the points $x\in \Q_2$ with
%$$|x|_2={1 \over \sqrt{|a|_2}}={1 \over 2^{k}}.$$
%Note that for all $x, y\in \Q_2$, with $|x|_2=|y|_2=2^{-k}$, we have
%\begin{align}\label{dist-estmates}
%\left|\phi(x)-\phi(y)\right|_2=\left|a-\frac{1}{xy}\right|_2\cdot |x-y|_2 \leq 2^{2k-1} |x-y|_2,
%\end{align}
%\[
%\phi(-2^k)=-{1 \over 2^k}-2^ka={a_{1-2k} \over 2^{k-1}} +{a_{2-2k} \over 2^{k-2}}+\cdots,
%\]
%and 
%\[
%\phi(2^k)={1 \over 2^k}+2^ka=-\left({a_{1-2k} \over 2^{k-1}} +{a_{2-2k} \over 2^{k-2}}+\cdots\right).
%\]
%
%If (i) $a_{1-2k}=1$, then $|\phi(-2^k)|_2=|\phi(2^k)|_2=2^{k-1}$. Thus the only two closed disks $D(-2^k, 2^{-k-2})$ and $D(2^k, 2^{-k-2})$
%of $A_2$ are sent into two disks of radius $2^{k-3}$ in the sphere $\{|x|_2=2^{k-1}\}$. Then $\phi(x) \in A_1$ and we have $\lim_{n\to\infty} \phi^n(x)=\infty$.
%
%If (ii) $a_{1-2k}=0$ and $a_{2-2k}=1$, then $|\phi(-2^k)|_2=|\phi(2^k)|_2=2^{k-2}$, Thus the only two closed disks $D(-2^k, 2^{-k-2})$ and $D(2^k, 2^{-k-2})$ of $A_2$
%are sent into the disk of radius $2^{k-3}$ in the sphere $\{|x|_2=2^{k-2}\}$. Then we still have $\phi(x) \in A_1$ and $\lim_{n\to\infty} \phi^n(x)=\infty$.

\end{proof}

Now we study the second subcase $\sqrt{1-a}\in \Q_2$. We first find the fixed points and then investigate the dynamical behaviors nearby the fixed points.
\begin{lemma}\label{lem3.2}
Suppose $|a|_2>1$. If $\sqrt{1-a}\in\Q_2$, then $\phi$ has two repelling fixed points $$x_{1,2}=\pm\frac{1}{\sqrt{1-a}}.$$
\end{lemma}

\begin{proof}
It is easy to check that $x_1=\frac{1}{\sqrt{1-a}}$ and $x_2=-\frac{1}{\sqrt{1-a}}$ are the two fixed points of $\phi$.
Note that  
$$\phi^{\prime}(x_{1})=\phi^{\prime}(x_{2})=2a-1.$$ 
Since $|a|_2>1$ and $\sqrt{1-a}\in\Q_2$, by   Lemma \ref{solution}, $v_2(a)$ is even, which implies $|a|_2\geq 4$.
Hence,  we have 
$$|\phi^{\prime}(x_{1})|_2=|\phi^{\prime}(x_{2})|_2=|2a-1|_2=|2a|_2>1.$$ 
Therefore, both fixed points are repelling.
\end{proof}

\begin{lemma}\label{scaling}
Suppose $|a|_2>1$. If $\sqrt{1-a}\in\Q_2$,  then
$$\phi(D(x_1,\frac{1}{4\sqrt{|a|_2}}))=\phi(D(x_2,\frac{1}{4\sqrt{|a|_2}}))=D(x_1, \frac{\sqrt{|a|_2}}{8})$$
and 
$$\forall x,y \in  D( x_i, \frac{1}{4\sqrt{|a|_2}}), \quad  |\phi(x)-\phi(y)|_2= \frac{|a|_2 |x-y|_2}{2}.$$
\end{lemma}
\begin{proof}
In Lemma \ref{lem3.2}, we have shown  $|a|_2\geq 4$. Since $|x_1-x_2|_2=\frac{1}{2\sqrt{|a|_2}}$, we have $x_2\in D(x_1, \frac{\sqrt{|a|_2}}{8}).$
Thus to prove Lemma \ref{scaling}, it suffices to show that 
$$|\phi(x)-\phi(y)|_2=\frac{|a|_2}{2}|x-y|_2$$
if  $x,y\in D(x_1,\frac{1}{4\sqrt{|a|_2}})$ or $D(x_2,\frac{1}{4\sqrt{|a|_2}})$.

Without loss of generality, we assume  $x,y \in D(x_1,\frac{1}{4\sqrt{|a|_2}})$.
Observe that  the disk $D(x_1,\frac{1}{4\sqrt{|a|_2}})$ is the image of $D({1 \over x_1},\frac{\sqrt{|a|_2}}{4})$ by the map $f(x)=1/x$, i.e. $$ D(x_1,\frac{1}{4\sqrt{|a|_2}})=\left\{\frac{1}{x}: |x-\sqrt{1-a}|_2 \leq \frac{\sqrt{|a|_2}}{4}\right\}.$$
Hence, there are $x^{\prime},y^{\prime}\in D(0,\frac{\sqrt{|a|_2}}{4})$  such that 
$$x=\frac{1}{\sqrt{1-a}+x^\prime} \ \  \hbox{ and } \ \  y=\frac{1}{\sqrt{1-a}+y^\prime}.$$
So  
$$\left|a-\frac{1}{xy}\right|_2=\left|2a-1+(x^{\prime}+y^{\prime})\sqrt{1-a}\right|_2.$$
Noting that $|a|_2\geq 4$ and  $|(x^{\prime}+y^{\prime})\sqrt{1-a}|_2 \leq |a|_2/4$, we have  
$$\left|a-\frac{1}{xy}\right|_2=\frac{|a|_2}{2}.$$
Hence, 
$$|\phi(x)-\phi(y)|_2=\left|(a-\frac{1}{xy})(x-y)\right|_2=\frac{|a|_2}{2}|x-y|_2.$$
 \end{proof}

The following lemma shows that the point $\infty$ is an attracting fixed point.
\begin{lemma}\label{attracting} Suppose $|a|_2>1$. If $\sqrt{1-a}\in\Q_2$, then 
 \[\ \quad \lim_{n\to\infty} \phi^n(x)=\infty.\]
 for all $x\in \Q_2\setminus(D(x_1,\frac{1}{4\sqrt{|a|_2}})\cup D(x_2,\frac{1}{4\sqrt{|a|_2}})).$
\end{lemma}

\begin{proof}
For $x\in \Q_2\setminus D(0,1/2)$,  we have$|1/x|_2\leq1$. Then $$|\phi(x)|_2=\left|ax+\frac{1}{x}\right|_2=|a|_2|x|_2>1.$$
Hence, $$|\phi^n(x)|_2=|a|^n_2|x|_2,$$ 
which implies 
\[\forall \ x\notin D(0,1/2), \quad \lim_{n\to\infty} \phi^n(x)=\infty.\]
Thus, to conclude it  suffices to show  that 
\[ \forall 0\neq x\in D(0,1)\setminus \left(D(x_1,\frac{1}{4\sqrt{|a|_2}})\cup D(x_2,\frac{1}{4\sqrt{|a|_2}})\right), \quad |\phi(x)|_2\geq 1.\]
By noting that $D(x_1,\frac{1}{4\sqrt{|a|_2}})\cup D(x_2,\frac{1}{4\sqrt{|a|_2}})=S(0,\frac{1}{\sqrt{|a|_2}})$,
we distinguish two cases.\\
Case (1) $|x|_2<1/\sqrt{|a|_2}$. Then   $|1/x|_2 > |ax|_2$, and thus 
$$|\phi(x)|_2=|1/x|_2>1.$$
Case (2) $1>|x|_2>1/\sqrt{|a|_2}$. Then $ |ax|_2>|1/x|_2>1$, and hence
$$|\phi(x)|_2= |ax|_2>1.$$
 \end{proof}

Now we are ready to show the dynamical structure of the system for the second subcase and finish the proof of Theorem \ref{thm-2}.
\begin{proposition}\label{prop-2}
Suppose $|a|_2>1$. If $\sqrt{1-a}\in\Q_2$, then $$\mathcal{J}_{\phi} =\left\{x\in \Q_p: |\phi^{n}(x)|_2=\frac{1}{\sqrt{|a|_2}} \ \hbox{ for all } n\geq 0\right\}$$
and 
$(\mathcal{J}_{\phi},\phi)$ is topologically  conjugate to $(\Sigma_2, \sigma)$.
Moreover,  $$\forall x\in \Q_2\setminus \mathcal{J}_\phi, \quad \lim_{n\to \infty}\phi^n(x)=\infty.$$
\end{proposition}
\begin{proof} Take $X=D(x_1,\frac{1}{4\sqrt{|a|_2}})\cup D(x_2,\frac{1}{4\sqrt{|a|_2}})$. Consider the restriction  map  $\phi: X\to \mathbb{P}^{1}(\Q_2)$. Define $\mathcal{K}_{\phi, X}$ as (\ref{Julia}).
By Theorem \ref{thm-repeller}, the $p$-adic repeller ($\mathcal{K}_{\phi, X}, \phi)$ is topologically conjugate to the full shift $(\Sigma_{2},\sigma)$ and   the points  which do not lie in $\mathcal{K}_{\phi, X}$  are sent to $\mathbb{P}^{1}(\Q_2)\setminus X.$ By Lemma \ref{attracting},  $$\forall x\in \Q_2\setminus \mathcal{K}_{\phi,X},  \quad \lim_{n\to \infty}\phi^n(x)=\infty.$$ 
Hence we have $\mathcal{J}_{\phi}= \mathcal{K}_{\phi,X}$.
\end{proof}

\begin{proof}[Proof of Theorem \ref{thm-2}] Combining Propositions \ref{prop-1} and  \ref{prop-2}, we complete the proof.
\end{proof}
%\subsection{$|a|_2=1$}
%
%\begin{theorem}
%
%\end{theorem}
\medskip
\subsection{$|a|_2<1$}
We begin with two lemmas which will be useful.
\begin{lemma}\label{largesphere}
If $|a|_2<1$, then 
\begin{equation*}
|\phi(x)-\phi(y)|_2 =
\begin{cases}
|a|_2|x-y|_2 , \hbox{ if  } x, y \in S(0,2^i) \hbox{ with }  i> v_2(a)/2;\\
2^{-2i}|x-y|_2,\hbox{ if }  x, y \in S(0,2^i) \hbox{ with }  i\leq 0.
\end{cases}
\end{equation*} 
Moreover, 
\begin{equation*}
\begin{cases}
\phi(S(0,2^i)) =  S(0,|a|_2 2^i), \hbox{ for  } i> v_2(a)/2;\\
\phi(S(0,2^i))=  S(0,2^{-i}),\hbox{ for } i\leq  0.
\end{cases}
\end{equation*}
\end{lemma}
\begin{proof}
Note that $$|\phi(x)-\phi(y)|_2=\left|a-\frac{1}{xy}\right|_2|x-y|_2.$$

If $ x, y \in S(0,2^i)$ with  $i> v_2(a)/2$, then  $|a|_2>\frac{1}{|xy|_2}$.
So 
$$|\phi(x)-\phi(y)|_2 =|a|_2|x-y|_2.$$  Moreover, $|ax|_2<|1/x|_2$. Thus
$|\phi(x)|_2=|ax+1/x|_2=|ax|_2$ which implies  $$\phi(S(0,2^i)) =  S(0,|a|_2 2^i).$$

If $ x, y \in S(0,2^i)$ with  $i\leq 0$, then $|a|_2<\frac{1}{|xy|_2}$. So
$$|\phi(x)-\phi(y)|_2 =2^{-2i}|x-y|_2.$$
Moreover, $|ax|_2<1\leq |1/x|_2$.  Hence
$|\phi(x)|_2=|ax+1/x|_2=|1/x|_2$ which implies  $$\phi(S(0,2^i))=  S(0,2^{-i}).$$
\end{proof}

\begin{lemma}\label{6}
If $|a|_2<1$, then 
for all $ - \lfloor  (v_{2}(a)-1)/2\rfloor \leq i \leq \lfloor  (v_{2}(a)-1)/2\rfloor$,  $$\phi(S(0,2^i))\subset   S(0,2^{-i})$$ and $\phi^{ 2}$ is $1$-Lipschitz continuous on   $S(0,2^i)\cup S(0,2^{-i})$.
\end{lemma}
\begin{proof}
If $x\in S(0,2^{i})$ for some $-\lfloor  (v_{2}(a)-1)/2\rfloor\leq  i \leq 0 $, then
by the assumption $|a|_2<1$, we have 
$$|\phi(x)|_2=|ax+1/x|_2=|1/x|_2=2^{-i}.$$
Now let $x\in S(0,2^{i})$ for some $0\leq  i \leq  \lfloor  (v_{2}(a)-1)/2\rfloor$.
When $i \leq  \lfloor  (v_{2}(a)-1)/2\rfloor$,  we have $$|x|_2=2^i \leq 2^{  \lfloor  (v_{2}(a)-1)/2\rfloor}\leq  2^{  (v_{2}(a)-1)/2}<2^{v_{2}(a)/2},$$  which implies $$|ax|_{2}<|1/x|_2.$$
Hence, the first assertion of the lemma holds.
%$$|\phi(x)|_2=|ax+1/x|_2=2^{-i}.$$

Let us show that  $\phi^{2}$ is $1$-Lipschitz continuous on   $S(0,2^i)\cup S(0,2^{-i})$ for $ - \lfloor  (v_{2}(a)-1)/2\rfloor \leq i \leq \lfloor  (v_{2}(a)-1)/2\rfloor$. 
Let $x,y \in S(0,2^{i})\cup  S(0,2^{-i})$.  If $x\in S(0,2^{i})$ and  $y \in S(0,2^{-i})$, then $|xy|_2=1$ and 
$$|\phi(x)-\phi(y)|_2=\left|a-\frac{1}{xy}\right|_2|x-y|_2=|x-y|_2.$$
Hence, it suffices to show that  for each $ - \lfloor  (v_{2}(a)-1)/2\rfloor \leq i \leq \lfloor  (v_{2}(a)-1)/2\rfloor$, $$ \forall x,y \in S(0,2^{i}), \quad |\phi^{ 2}(x)-\phi^{ 2}(y)|_2\leq |x-y|_2.$$
By observing that  $\frac{1}{|xy|_2}=2^{-2i}> |a|_2$ and $\frac{1}{|\phi(x)\phi(y)|_2}=2^{2i}>|a|_2$, we have 
\begin{align*}  
|\phi^{ 2}(x)-\phi^{ 2}(y)|_2&=  \left|a-\frac{1}{\phi(x)\phi(y)}\right|_2|\phi(x)-\phi(y)|_2\\
&=\left|a-\frac{1}{\phi(x)\phi(y)}\right|_2 \left|a-\frac{1}{xy}\right|_2|x-y|_2\\
& \leq \left|\frac{1}{\phi(x)\phi(y)}\right|_2 \left|\frac{1}{xy}\right|_2|x-y|_2=|x-y|_2.
\end{align*}

\end{proof}

Now we distinguish two subcases: $\sqrt{-a}\in \Q_2$ and $\sqrt{-a}\notin \Q_2$.
\subsubsection{$\sqrt{-a}\in \Q_2$}

\begin{lemma}\label{gotozero}
If $\sqrt{-a}\in \Q_2$, then  
$$\phi\left( D( {1 \over \sqrt{-a}}, \frac{1}{4\sqrt{|a|_2}})\right)=\phi\left( D( -{1 \over \sqrt{-a}}, \frac{1}{4\sqrt{|a|_2}})\right)=D(0,\frac{\sqrt{|a|_2}}{8}),$$ and  
$$\forall x,y \in  D( \pm{1 \over \sqrt{-a}}, \frac{1}{4\sqrt{|a|_2}}), \quad  |\phi(x)-\phi(y)|_2= \frac{|a|_2 |x-y|_2}{2}.$$
\end{lemma}

\begin{proof}
Note that  $ \phi(\pm 1/\sqrt{-a})=0$. It suffices to show that 
$$\forall x, y \in D( \pm {1 \over \sqrt{-a}}, \frac{1}{4\sqrt{|a|_2}}),\quad |\phi(x)-\phi(y)|_2=  \frac{|a|_2 |x-y|_2}{2}.$$ 
Without loss of generality, we assume $x,y \in D( {1 \over \sqrt{-a}}, \frac{1}{4\sqrt{|a|_2}})$. By the same arguments in the proof of Lemma \ref{scaling}, there exist $x^{\prime},y^{\prime}\in D(0, \frac{\sqrt{|a|_2}}{4})$
such that 
$$x=\frac{1}{\sqrt{-a}+x^{\prime}},  \quad y=\frac{1}{\sqrt{-a}+y^{\prime}},$$
and
$$\left|a-\frac{1}{xy}\right|_2= \left|2a-(x^{\prime}+y^{\prime})\sqrt{-a}-x^{\prime}y^{\prime}\right|_2.$$
Since $x^{\prime},y^{\prime}\in D(0, \frac{\sqrt{|a|_2}}{4})$, we immediately get 
 $$\left|a-\frac{1}{xy}\right|_2=\frac{|a|_2}{2}.$$
Thus
$$|\phi(x)-\phi(y)|_2=\left|a-\frac{1}{xy}\right|_2 |x-y|_2=\frac{|a|_2 |x-y|_2}{2}.$$
\end{proof}

\begin{proposition}\label{Prop3.9}
Suppose $|a|_2<1$ and $\sqrt{-a}\in\Q_2$.   Let $\Omega=\PP\setminus ( S(0, \sqrt{|a|_2}) \cup X_a) $ and consider the restriction map   $\phi: \Omega \to \PP$.  
Set $\mathcal{K}_{\phi, \Omega}=\bigcap_{n=0}^\infty \phi^{-n}(\Omega)$. We distinguish two cases.\\
{\rm (i)} If $v_{2}(a)\neq 2$, then the subsystem $\left(\mathcal{K}_{\phi, \Omega}, \phi \right)$ is topologically conjugate to a subshift of finite type $(\Sigma_{A},\sigma),$
where the incidence matrix is
$$A=\begin{pmatrix} 0 & 0& 1 & 0   \\ 0&0&1&0\\  0&0&0&1 \\  1&1&0&1 \end{pmatrix}.$$\\
{\rm (ii)} If $v_{2}(a)= 2$, then the subsystem $\left(\mathcal{K}_{\phi, \Omega}, \phi \right)$ is topologically conjugate to a subshift of finite type $(\Sigma_{A},\sigma),$
where the incidence matrix is
$$A=\begin{pmatrix} 0 & 0& 1 & 0 &0  \\ 0&0&1&0&0\\  0&0&0&1&0 \\  0&0&0&1&1 \\ 1&1&0&0&0 \end{pmatrix}.$$
\end{proposition}
\begin{proof}
Let $D_1= D( {1 \over \sqrt{-a}}, \frac{1}{4\sqrt{|a|_2}}),$ $ D_2= D( -{1 \over \sqrt{-a}}, \frac{1}{4\sqrt{|a|_2}})$, 
$D_3=D(0,\frac{\sqrt{|a|_2}}{8})$   and $D_4=\P\setminus D(0,\frac{4}{\sqrt{|a|_2}})$. 
By Lemma \ref{gotozero}, the restricted maps $\phi:D_1 \to D_3$ and  $\phi:D_2 \to D_3$  are  both  scaling bijective. 
One can also directly  check  that  $\phi:D_3\to D_4$ and  $\phi: D_4\to \phi(D_4)=\P\setminus D(0,4\sqrt{|a|_2})$ are  
bijective. \\
{\rm (i)} Since $|a|_2<4$, we have $D_1\cup D_2 \subset \phi(D_4)$.  Let  $X=\cup_{i=1}^{4}D_i$. Consider the map $\phi:X\to \PP$. Set  $$\mathcal{K}_{\phi, X}=\bigcap_{n=0}^\infty \phi^{-n}(X).$$ 
Let $g: x\mapsto1/(x-1)$ and $\psi:= g\circ\phi\circ g^{-1}$. 
We  study the map $\psi:  g(X)\to \P$. Observe $g(X)\subset \mathbb{Z}_p$,
%Take $$\mathcal{J^{\prime}}=\bigcap_{i=0}^\infty(\psi^{-1})^{\circ i}(X).$$ 
and that $(\mathcal{K}_{\psi, g(X)}, \psi)$ is a generalized weak repeller with incidence matrix
 $$A=\begin{pmatrix} 0 & 0& 1 & 0   \\ 0&0&1&0\\  0&0&0&1 \\  1&1&0&1 \end{pmatrix}.$$ Thus, by Theorem \ref{general},  
$(\mathcal{K}_{\psi, g(X)}, \psi)$ is topologically conjugate to the subshift of finite type $(\Sigma_{A},\sigma)$. Take $\mathcal{K}_{\phi, X}=g^{-1}(\mathcal{K}_{\psi, g(X)})$.
Since $\psi= g\circ\phi\circ g^{-1}$, we deduce that
$(\mathcal{K}_{\phi, X}, \phi)$ is topologically conjugate to $(\Sigma_{A},\sigma).$ 

Noting that  $$\Omega\setminus X:=S(0,\frac{\sqrt{|a|_2}}{2})\cup S(0,\frac{\sqrt{|a|_2}}{4})\cup S(0,\frac{2}{\sqrt{|a|_2}})\cup S(0,\frac{4}{\sqrt{|a|_2}}),$$ by Lemmas  \ref{largesphere} and \ref{6}, we have $\phi(\Omega\setminus X) \subset X_a$, since $v_2(a)>2$. Thus we have $\mathcal{K}_{\phi, X}=\mathcal{K}_{\phi, \Omega}$.

\smallskip
\noindent {\rm (ii)} Let $D_5= S( 0, {4 \over \sqrt{|a|_2}} )$.    Since $|a|_2=4$, we have $D_5 \subset \phi(D_4)$.  By  Lemma \ref{largesphere},  the map $\phi:D_5\to D_1\cap D_2$ satisfies 
$$\forall x,y \in D_5, \quad |\phi(x)-\phi(y)|_2 =|a|_2|x-y|_2. $$   Let  $X=\cup_{i=1}^{5}D_i$ and consider the map $\phi:X\to \PP$.  By the same argument   as the case (i),  we have   $\mathcal{K}_{\phi, X}=\mathcal{K}_{\phi, \Omega}$ and $(\mathcal{K}_{\phi, X},\phi)$ is topologically conjugate to a subshift of finite type $(\Sigma_{A},\sigma),$
with incidence matrix 
$$A=\begin{pmatrix} 0 & 0& 1 & 0 &0  \\ 0&0&1&0&0\\  0&0&0&1&0 \\  0&0&0&1&1 \\ 1&1&0&0&0 \end{pmatrix}.$$
\end{proof}

\subsubsection{$\sqrt{-a}\notin \Q_2$}

\begin{lemma}\label{3.9}
If  $v_2(a)$ is odd, then for each $x\in \Q_2\setminus \{0\}$, there exists a  positive integer $N$ such that 
$$\phi^{ n}(x)\in S(0,2^i)\cup S(0,2^{-i}), \quad \forall n\geq N,$$   for some $0\leq i\leq  \frac{v_2(a)-1}{2}$.
\end{lemma}
\begin{proof}
Note that $|\phi(x)|_2=|{1}/{x}|_2$, if $|x|_2<  2^{- (v_2(a)-1)/2}$. Thus we have
$$\phi(D(0,2^{-( v_2(a)-1)/2-1}))=\mathbb{P}^{1}(\Q_2)\setminus D(0,2^{  (v_2(a)-1)/2}).$$
By Lemma \ref{6}, it suffices to show that the statement holds for the points in $\mathbb{P}^{1}(\Q_2)\setminus  D(0,2^{ (v_2(a)-1)/2})$.
 In fact, if $x\in  \mathbb{P}^{1}(\Q_2)\setminus D(0,2^{ (v_2(a)-1)/2}) $,  one can check that 
 $$|\phi(x)|_2=|a|_2|x|_2\leq |x|_2.$$
 So there exists an integer $N$ such that $$2^{- ( v_2(a)-1)/2}\leq |\phi^{ N}(x)|_2\leq 2^{  (v_2(a)-1)/2}.$$  By  Lemma \ref{6}, we conclude.
\end{proof}

\begin{lemma}\label{3.10}
If   $v_2(a)>0$ is even  and  $\sqrt{-a}\notin \Q_2(\sqrt{-3})$, then $$\phi( S(0, \frac{1}{\sqrt{|a|_2}})\subset  S(0, \frac{\sqrt{|a|_2}}{2}).$$
\end{lemma}
\begin{proof}
 For $ x\in S(0, \frac{1}{\sqrt{|a|_2}})$, we have $x=2^{-\frac{v_2(a)}{2}}+y$ for some $y\in D(0,\frac{1}{2\sqrt{|a|_2}})$. 
 Noting that $v_2(a)$ is even, we deduce from $\sqrt{-a}\notin \Q_2(\sqrt{-3})$ that
$$2^{v_2(a)}a \equiv1 \hbox{ or }-3 {\rm  ~(\!\!\!\!\!\mod 8)}.$$  Thus $$|\phi(x)|_2=\frac{|ax^2+1|_2}{|x|_2}= \frac{|2^{-v_2(a)}a+1+2\cdot a \cdot 2^{-v_2(a)/2}y+a\cdot y^2|_2 }{|x|_2}=\frac{1}{2|x|_2}.$$
Therefore,  $$\phi( S(0, \frac{1}{\sqrt{|a|_2}})\subset  S(0, \frac{\sqrt{|a|_2}}{2}).$$
\end{proof}

\begin{lemma}\label{3.11}
If  $\sqrt{-a}\in \Q_2(\sqrt{-3})$, then  
$$\phi( D(2^{-\frac{v_2(a)}{2}}, \frac{1}{4\sqrt{|a|_2}}))=\phi(  D(-2^{-\frac{v_2(a)}{2}}, \frac{1}{4\sqrt{|a|_2}}))=S(0,\frac{\sqrt{|a|_2}}{4}),$$ and  
$$\forall x,y \in  D(\pm \frac{1}{\sqrt{|a|_2}}, \frac{1}{4\sqrt{|a|_2}}), \quad |\phi(x)-\phi(y)|_2= \frac{|a|_2 |x-y|_2}{2}.$$
\end{lemma}
\begin{proof}
Without loss of generality, we assume $x,y \in   D(2^{-\frac{v_2(a)}{2}}, \frac{1}{4\sqrt{|a|_2}})$. By the same arguments in the proof of Lemma \ref{scaling}, there exist $x^{\prime},y^{\prime}\in D(0, \frac{\sqrt{|a|_2}}{4})$
such that 
$$x=\frac{1}{2^{-\frac{v_2(a)}{2}}+x^{\prime}},  \quad y=\frac{1}{2^{-\frac{v_2(a)}{2}}+y^{\prime}},$$
and
$$\left|a-\frac{1}{xy}\right|_2= \left|a-2^{-v_2(a)}-(x^{\prime}+y^{\prime})2^{-\frac{v_2(a)}{2}}-x^{\prime}y^{\prime}\right|_2.$$
Note that  $\sqrt{-a}\in \Q_2(\sqrt{-3})$ implies $2^{v_2(a)} a\equiv 3~{\rm   (\!\!\!\!\!\mod 8)}$.
Since $x^{\prime},y^{\prime}\in  D(0, \frac{\sqrt{|a|_2}}{4})$, we immediately get 
 $$\left|a-\frac{1}{xy}\right|_2=\frac{|a|_2}{2}.$$
Hence
$$|\phi(x)-\phi(y)|_2=\left|a-\frac{1}{xy}\right|_2 |x-y|_2=\frac{|a|_2 |x-y|_2}{2}.$$

\end{proof}

\begin{lemma}\label{3.12}
Assume that  $v_2(a)$ is even. If $v_2(a)\neq 2$ or $\Q_2(\sqrt{-a})\neq \Q_2(\sqrt{-3})$,
then for each $x\in \Q_2\setminus \{0\}$, there exists a  positive integer $N$ such that 
$$ \forall n\geq N, \quad \phi^{ n}(x)\in S(0,2^i)\cup S(0,2^{-i}),$$   for some $0\leq i\leq  \frac{v_2(a)}{2}-1$.
\end{lemma}
\begin{proof}
%We remind that by the assumptions that $|a|_2<1$, $v_2(a)$ is even, $v_2(a)\neq 2$, we have $|a|_2 \leq 1/16$.

If $|x|_2 \leq   \sqrt{|a|_2}$, then $$|\phi(x)|_2=|{1/x}|_2 \geq {1 \over \sqrt{|a|_2}}.$$ Thus we have
$$\phi(D(0,\sqrt{|a|_2}))\subset \mathbb{P}^{1}(\Q_2)\setminus D(0,\frac{1}{2\sqrt{|a|_2}}).$$
By  Lemma \ref{6}, it suffices to show that the statement holds for all points   in $\mathbb{P}^{1}(\Q_2)\setminus D(0,\frac{1}{2\sqrt{|a|_2}})$.
 In fact, if $x\in  \mathbb{P}^{1}(\Q_2)\setminus D(0,\frac{1}{\sqrt{|a|_2}}) $, by noting $|a|_2 \leq 1/16$, one can check that 
 $$|\phi(x)|_2=|a|_2|x|_2<|x|_2.$$
 Hence there exists an integer $N$ such that $$2^{-v_2(a)/2}< |\phi^{ N}(x)|_2\leq 2^{  v_2(a)/2}.$$ Therefore, it  remains to show that the statement holds for all $x \in S( 0,\frac{1}{\sqrt{|a|_2}})$.  We distinguish two cases.\\
Case (1)  $v_2(a)\neq 2$. 
 By Lemma \ref{3.11}, we have $$\phi( S(0, \frac{1}{\sqrt{|a|_2}})\subset  S(0, \frac{\sqrt{|a|_2}}{4}).$$ 
 Observing $\phi(S(0, \frac{\sqrt{|a|_2}}{4}))\subset   S(0, \frac{4}{\sqrt{|a|_2}})$  and  $\phi(S(0, \frac{4}{\sqrt{|a|_2}}))\subset S(0, 4\sqrt{|a|_2}) $, we obtain
  $$\phi^{3}(S(0, \frac{1}{\sqrt{|a|_2}}) \subset  S(0, 4\sqrt{|a|_2}).$$ 
  Since $v_2(a)\neq 2$ , we have   $$ -\frac{v_2(a)}{2}-1\leq -\frac{v_2(a)}{2} +2= 4\sqrt{|a|_2}\leq  \frac{v_2(a)}{2}-1.$$ 
  Hence, we conclude by Lemma \ref{6} again.

\smallskip
\noindent Case (2)  $\sqrt{-a} \in \Q_2(\sqrt{-1})$ or $\Q_2(\sqrt{3})$.
  By Lemma \ref{3.10},  we have $$\phi( S(0, \frac{1}{\sqrt{|a|_2}})\subset  S(0, \frac{\sqrt{|a|_2}}{2}).$$ Since $\phi(S(0, \frac{\sqrt{|a|_2}}{2}))\subset   S(0,\frac{2}{\sqrt{|a|_2}})$  and  $\phi(S(0, \frac{2}{\sqrt{|a|_2}}))\subset S(0, 2^{ -\frac{v_2(a)}{2} +1}) $, we have
  $$\phi^{3}(S(0, \frac{1}{\sqrt{|a|_2}}) \subset  S(0, 2^{ -\frac{v_2(a)}{2} +1}).$$ 
  Applying Lemma \ref{6}, we finish the proof.
\end{proof}

\begin{proposition}\label{-3}
Assume that    $|a|_2<1$ and  $ \sqrt{-a}  \notin \Q_2$. 
 If  $\sqrt{-a} \in \Q_2(\sqrt{-3})$ and $v_2(a)=2$,   then the subsystem $(S(0,2^{3})\cup S(0,2^{-3} ) \cup S(0,2), \phi)$  is topologically conjugate to a subshift of finite type $(\Sigma_{A},\sigma),$
with incidence matrix 
$$A=\begin{pmatrix} 0 & 0& 1 & 0   \\ 0&0&1&0\\  0&0&0&1 \\  1&1&0&0 \end{pmatrix}.$$ 
%and  for each $x\in \Q_2\setminus (\{0\} \cup \mathcal{J}_\phi)$, there exists an integer  $N$ such that 
%$$\phi^{ n}(x)\in S(0,2^i)\cup S(0,2^{-i}), \quad \forall n\geq N,$$   for some $0\leq i\leq  \frac{v_2(a)}{2}-1$.
% \\
%{\rm (2)} Otherwise, $\mathcal{J}_{\phi}=\{0,\infty\}.$  Moreover, for each  $x\in \Q_2\setminus \{0\} $, there exists an integer $N$ such that 
%$$\phi^{ n}(x)\in S(0,2^i)\cup S(0,2^{-i}), \quad \forall n\geq N,$$   for some $0\leq i\leq   \lfloor  (v_{2}(a)-1)/2 \rfloor$.
\end{proposition}
\begin{proof}
%We remind that by the assumptions that $|a|_2<1$, $v_2(a)$ is even, $v_2(a)= 2$, we have $|a|_2 =1/4$.
Let $D_1= D( {1 \over 2}, {1 \over 2})$,  $ D_2= D(  {-1 \over 2},{1 \over 2})$, 
$D_3=D(8,\frac{1}{16})$   and $D_4=D(\frac{1}{8},4)$. 
By Lemma \ref{3.11}, the restricted maps $\phi:D_1 \to D_3$ and  $\phi:D_2 \to D_3$  are both bijective. 
Since $\phi(x)-\phi(y)=(a-\frac{1}{xy})(x-y)$, we have $$|\phi(x)-\phi(y)|_2= 64 |x-y|_2 \ \hbox{ if } x,y \in D_3$$
and $$   |\phi(x)-\phi(y)|_2=\frac{|x-y|_2}{4} \ \hbox{ if }x,y \in D_4.$$
So   $\phi:D_3\to D_4$ and  $\phi: D_4\to \phi(D_4)=D_1\cup D_2$ are  
bijective. 

Set  $\Omega=\bigcup_{i=1}^{4}D_i$ and consider the restricted map $\phi: \Omega \to \P$.   Take $$\mathcal{K}_{\phi, \Omega}=\bigcap_{n=0}^\infty\phi^{-n}(\Omega)=\Omega.$$ 
By Theorem \ref{general}, $(\Omega, \phi)$  is topologically conjugate to a subshift of finite type $(\Sigma_{A},\sigma)$
with incidence matrix 
$$A=\begin{pmatrix} 0 & 0& 1 & 0   \\ 0&0&1&0\\  0&0&0&1 \\  1&1&0&0 \end{pmatrix}.$$
%\\
%{\rm (2)}By  Lemmas \ref{3.9} and \ref{3.12}, for each  $x\in \Q_2\setminus \{0\} $, there exists an integer $N$ such that 
%$$\phi^{ n}(x)\in S(0,2^i)\cup S(0,2^{-i}), \quad \forall n\geq N,$$   for some $0\leq i\leq   \lfloor  (v_{2}(a)-1)/2 \rfloor$. By Lemma \ref{6}, we have $\Q_p\setminus \{0\}  \in \mathcal{F}_{\phi}$. Noting that $\phi(0)=\infty$ is a repelling fixed point, we have 
%$\mathcal{J}_{\phi}=\{0,\infty\}.$ 
\end{proof}
\begin{proof}[Proof of Theorem \ref{a<1}]
By Lemma \ref{6}, we immediately  have $\phi(X_a)\subset X_a$ and  $X_a \subset \mathcal{F}_\phi$.  By the argument of the proof of Proposition 3 in \cite{FL16},  we obtain  a minimal decomposition for  the subsystem $(X_a,\phi)$  as stated  in Theorem \ref{minid}.
It remains to show that for each  $x\in \mathcal{F}_\phi$, there exists some positive integer $N$ such that  $\phi^{N}(x)\in X_a$. We distinguish the following three cases.

\medskip
{\rm (i)} Let $\Omega=\PP\setminus ( S(0, \sqrt{|a|_2}) \cup X_a) $.  By Lemmas \ref{largesphere} and \ref{gotozero}, we have 
\begin{align}\label{phiomega}
\phi(\Omega)\subset \PP\setminus S(0, \sqrt{|a|_2}). 
\end{align}  
By Proposition \ref{Prop3.9}, $$\mathcal{K}_{\phi,\Omega}\subset \mathcal{J}_\phi.$$ 
By (\ref{phiomega}) and the definition of $\mathcal{K}_{\phi,\Omega}$, each  point $x\in \Omega\setminus  \mathcal{K}_{\phi,\Omega}$ goes to $X_a$ by iteration of $\phi$. Thus, $$\mathcal{J}_{\phi}\cap (\PP\setminus( S(0,\sqrt{|a|_2})\cap X_a))=\mathcal{K}_{\phi,\Omega}.$$
By Lemma \ref{largesphere}, $\phi(S(0,\sqrt{|a|_2}))=S(0,\frac{1}{\sqrt{|a|_2}})\subset \Omega $.  Hence,   $$\mathcal{J}_\phi=\left\{x\in \PP: \phi^{n}(x) \notin  X_a
 \hbox{  for all } n\geq 0 \right\}.$$ The dynamical behaviors of $(\mathcal{K}_{\phi,\Omega},\phi)$ have already been shown in Proposition \ref{Prop3.9}. 
 
 \medskip
{\rm (ii)} By Proposition \ref{-3},  $$S(0,2^{3})\cup S(0,2^{-3} ) \cup S(0,2)\subset \mathcal{J}_\phi.$$ By Lemma \ref{largesphere}, for each integer $k\geq 1$,  we have
\begin{equation*}\label{equ1}
\begin{cases}
\phi^{k-1}(x) \in  S(0,2^{3}), \hbox{ if } x\in  S(0,2^{2k+1}) ;\\
\phi^{k}(x) \in  S(0,2^{3}),\hbox{ if } x\in  S(0,2^{-(2k+1)}).
\end{cases}
\end{equation*} 
Thus, $$\bigcup_{k\geq 0} (S(0,2^{2k+1})\cup S(0,2^{-2k-1} )) \subset \mathcal{J}_{\phi}.$$  
However, if $$x\in \Q_p\setminus \left( \Big(\bigcup_{k\geq 0}\big(S(0,2^{2k+1})\cup S(0,2^{-2k-1} )\big)\Big)\cup \{0\} \right), $$ by Lemma \ref{largesphere}, there exists some $N$ such that  $\phi^{N}(x)\in X_a= S(0,1)$.  Noting that $\phi(0)=\infty$ is a repelling fixed point, we have 
$$\mathcal{J}_{\phi}=\{0,\infty \} \bigcup_{k\geq 0} (S(0,2^{2k+1})\cup S(0,2^{-2k-1} )).$$
The dynamical behaviors of $(S(0,2^{3})\cup S(0,2^{-3} ) \cup S(0,2),\phi)$ have already been shown in Proposition \ref{-3}. 

\medskip
{\rm (iii)} By  Lemmas \ref{3.9} and \ref{3.12}, for each  $x\in \Q_2\setminus \{0\} $, there exists an integer $N$ such that 
$$\forall n\geq N, \quad \phi^{ n}(x)\in S(0,2^i)\cup S(0,2^{-i}),$$   for some $0\leq i\leq   \lfloor  (v_{2}(a)-1)/2 \rfloor$. By Lemma \ref{6}, we have $\Q_p\setminus \{0\}  \in \mathcal{F}_{\phi}$. Noting that $\phi(0)=\infty$ is a repelling fixed point, we have 
$\mathcal{J}_{\phi}=\{0,\infty\}.$ 
\end{proof}
\medskip

\bibliographystyle{plain}
%\bibliography{ref}

\end{document}